\documentclass[11pt]{article}

\usepackage{amssymb, amsmath, fullpage, amsthm}

\newtheorem{theorem}{Theorem}[section]

\newtheorem{lemma}[theorem]{Lemma}

\newtheorem{corollary}[theorem]{Corollary}

\numberwithin{equation}{section}

\newcommand{\pair}[2]{\left(#1,#2\right)}
\newcommand{\wt}{\operatorname{wt}}
\newcommand{\Wt}{\operatorname{Wt}}
\newcommand{\bb}{{\mathbf b}{\mathbf b}}
\newcommand{\bbc}{{\mathbf b}{\mathbf b}^{c}}
\newcommand{\trace}{\operatorname{tr}}
\newcommand{\erf}{\operatorname{erf}}

\newcommand{\onethingatopanother}[2]{\genfrac{}{}{0pt}{}{#1}{#2}}

\begin{document}

\title{Cyclically consecutive permutation avoidance}
\author{Richard Ehrenborg}
\date{}

\maketitle

\begin{abstract}
We give an explicit
formula for the number of
permutations avoiding cyclically a consecutive pattern
in terms of the spectrum of
the associated operator of the consecutive pattern.
As an example,
the number of 
cyclically consecutive $123$-avoiding 
permutations in ${\mathfrak S}_{n}$
is given by
$n!$ times 
the convergent series
$\sum_{k=-\infty}^{\infty} \left(\frac{\sqrt{3}}{2\pi(k+1/3)}\right)^{n}$
for $n \geq 2$.
\end{abstract}

{\small
}

\section{Introduction}

The Euler number $E_{n}$
enumerates the number of alternating permutations
in the symmetric group~${\mathfrak S}_{n}$.
The Euler number has the following classical expression
\begin{equation}
    E_{n}
=
    n! 
\cdot
2
\cdot
\left(\frac{2}{\pi}\right)^{n+1}
\cdot
\sum_{k=-\infty}^{\infty}
\frac{1}{(4k+1)^{n+1}}   .
\label{equation_Euler}
\end{equation}
Simons and Yao~\cite{Simons_Yao} were the first
to the author's knowledge to give a probabilistic--geometric
argument for this identity.
See also~\cite[Section~4]{Ehrenborg_Levin_Readdy}
and~\cite{Elkies}.
The underlying idea is that
to obtain a random permutation~$\pi$ 
where each permutation is equally likely,
pick a random point
$x = (x_{1}, x_{2}, \ldots, x_{n})$
from the cube~$[0,1]^{n}$.
Next let the permutation $\pi$ be 
the standardization of the vector $x$,
that is, 
$\pi_{i} < \pi_{j}$ is equivalent to $x_{i} < x_{j}$
for all indexes $i \neq j$.
Geometrically, this corresponds to pick the permutation
associated with the simplex containing the point $x$
in the standard triangulation of the cube $[0,1]^{n}$.
Hence the probability that 
the point is alternating, that is,
$\cdots < x_{n-2} > x_{n-1} < x_{n}$
is given by $E_{n}/n!$.
By considering the distribution of the last coordinate $f_{n}$
of the point $x$
one obtains 
the recursion $f_{n} = T(f_{n-1})$
where $T$ is the operator 
given by $T(f) = \int_{0}^{x}  f(1-t) \: dt$.
This recursion is straightforward to solve using Fourier series.
Integrating this series yields
equation~\eqref{equation_Euler}.

Ehrenborg, Kitaev and Perry~\cite{Ehrenborg_Kitaev_Perry}
continued this work by studying consecutive pattern avoidance
using operators on the space $L^{2}([0,1]^{m})$.
The largest eigenvalue of the associated operator
yields the asymptotic behavior for the number of permutations.
In fact, more knowledge of the eigenvalues gives a better
asymptotic expansion;
see~\cite[Theorem~1.1]{Ehrenborg_Kitaev_Perry}.
However, a complete description of the eigenvalues
is hard to find. It is is only known in three cases.
First, consecutive $123,321$-avoiding permutations,
this is, the Euler numbers described above (modulo
a factor of $2$).
Second, consecutive $123$-avoiding permutations;
see Section~5 of~\cite{Ehrenborg_Kitaev_Perry}.
Finally, in~\cite[Section~6]{Ehrenborg_Jung} a weighted enumeration
problem is introduced where the associated
operator only has the eigenvalue $1$.

In Elkies' argument~\cite{Elkies}
for the series of the Euler number, he notes that the associated
operator is trace class and connects this with
cyclically alternating permutations.
We show that in the more general 
setting of consecutive pattern avoidance
if the operator has powers
that are trace class, there is an exact series
enumerating the cyclically avoiding permutations.
We explicitly state the result for
the number of cyclically consecutive $123$-avoiding
permutations; see
Theorem~\ref{theorem_123}.
Similarly for the weighted enumeration problem, we obtain
a closed form formula of $n!$;
see Theorem~\ref{theorem_weight}.

We end with open questions for further research.

\section{Weighted enumeration}

For a vector $x = (x_{1}, \ldots, x_{k})$
of $k$ distinct real numbers,
define $\Pi(x)$ to be {\em standardization} of the vector $x$,
that is, the unique permutation
$\sigma = (\sigma_{1}, \ldots, \sigma_{k})$
in ${\mathfrak S}_{k}$ such that
for all indices $1 \leq i < j \leq k$
the inequality
$x_{i} < x_{j}$ is equivalent to
$\sigma_{i} < \sigma_{j}$.
Let $S$ be a set of permutations in
the symmetric group ${\mathfrak S}_{m+1}$.
We say that a permutation $\pi$ in 
${\mathfrak S}_{n}$ {\em avoids the set $S$ consecutively}
if there is no index $1 \leq j \leq n-m$ such that
$\Pi(\pi_{j}, \pi_{j+1}, \ldots, \pi_{j+m}) \in S$.
Similarly,
a permutation $\pi \in {\mathfrak S}_{n}$
{\em avoids the set $S$ cyclically consecutively}
if there is no index $1 \leq j \leq n$ such that
$\Pi(\pi_{j}, \pi_{j+1}, \ldots, \pi_{j+m}) \in S$,
where the indexes are modulo $n$.

We consider the more general situation of
weighted enumeration of 
consecutive patterns; see~\cite{Ehrenborg_Jung}.
Let $\wt$ be
a real-valued weight function on the symmetric group
${\mathfrak S}_{m+1}$.
Similarly, let $\wt_{1}, \wt_{2}$ be
two real-valued weight functions on the symmetric group
${\mathfrak S}_{m}$.
We call $\wt_{1}$ and $\wt_{2}$
the initial and final weight function, respectively.
We extend these three weight functions to
the symmetric group~${\mathfrak S}_{n}$ for $n \geq m$ by
defining
\begin{align*}
     \Wt(\pi)
   = &
       \wt_{1}(\Pi(\pi_{1}, \pi_{2}, \ldots, \pi_{m}))  \\
      &
             \cdot
           \prod_{i=1}^{n-m}
       \wt(\Pi(\pi_{i}, \pi_{i+1}, \ldots, \pi_{i+m})) \\
       &
             \cdot
       \wt_{2}(\Pi(\pi_{n-m+1}, \pi_{n-m+2}, \ldots, \pi_{n})) .
\end{align*}
Let $\alpha_{n}$ be the sum of all the weights of
permutations in ${\mathfrak S}_{n}$, that is,
$$ \alpha_{n}
       =
   \sum_{\pi \in {\mathfrak S}_{n}} \Wt(\pi)  . $$
This framework can be used to study
consecutive pattern avoidance by defining the weight
$$  \wt(\sigma)
=
\begin{cases}
1 & \text{if } \sigma \not\in S , \\
0 & \text{if } \sigma \in S .
\end{cases} $$
Furthermore, let both the initial and final weight functions
be the constant function ${\mathbf 1}$.
Then for $n \geq m$, $\alpha_{n}$ is the number of permutations
in ${\mathfrak S}_{n}$ that avoid the set $S$.

Define the function $\chi$ on the $(m+1)$-dimensional cube
$[0,1]^{m+1}$ by
$\chi(x_{1}, x_{2}, \ldots, x_{m+1})
     =
       \wt(\Pi(x_{1},x_{2}, \ldots, x_{m+1}))$.
The associated linear operator $T$
on the space $L^{2}([0,1]^{m})$ is defined by
\begin{equation}
      T(f)(x_{1}, \ldots, x_{m})
     =
        \int_{0}^{1}
            \chi(t,x_{1}, \ldots, x_{m}) \cdot f(t,x_{1}, \ldots, x_{m-1}) \: dt .
\label{equation_T}
\end{equation}
Similar to $\chi$ define
the two functions $\kappa$ and $\mu$ on
the $m$-dimensional unit cube $[0,1]^{m}$ by
$\kappa(x)  =  \wt_{1}(\Pi(x))$
and
$\mu(x)  =  \wt_{2}(\Pi(x))$.
Then the quantity $\alpha_{n}$ is given by the
inner product
$$ \pair{T^{n-m}(\kappa)}{\mu}
      =
      \alpha_{n}/n!    .    $$
Hence to study the asymptotic expansion
of $\alpha_{n}$ one has to obtain the spectrum
of the operator $T$.

We now turn our attention to the power $T^{n}$
where $n \geq m$. Begin by expanding $\chi$
by the product 
$$  \chi_{n}(x_{1}, x_{2}, \ldots, x_{n})
        =
       \prod_{i=1}^{n-m}
            \chi(x_{i}, x_{i+1}, \ldots, x_{i+m}) . $$
Define the kernal $K_{n}(x,y)$ for $x,y \in [0,1]^{m}$
for $x = (x_{1}, \ldots, x_{m})$ and $y = (y_{1}, \ldots, y_{m})$
by the integral
$$  K_{n}(x,y) = \int_{[0,1]^{n-m}}
        \chi_{n+m}(x_{1}, \ldots, x_{m}, t_{1}, \ldots, t_{n-m},
                    y_{1}, \ldots, y_{m}) \: dt_{1} \cdots dt_{n-m}  .  $$
Since
\begin{equation}
   T^{n}f(y)  =  \int_{[0,1]^{m}}  K_{n}(x,y) dx  .  
\end{equation}
we conclude that $T^{n}$ is a Hilbert--Schmidt operator,
and hence $T^{n}$ is compact.

\section{Weighted cyclic enumeration}

Define the {\em cyclic weight} of a permutation $\pi$
in ${\mathfrak S}_{n}$, where $n \geq m+1$, by the product
$$  \Wt^{c}(\pi) 
   =
         \prod_{i=1}^{n}
       \wt(\Pi(\pi_{i}, \pi_{i+1}, \ldots, \pi_{i+m})) , $$
where the indexes are modulo $n$.
Similar to $\alpha_{n}$,
we consider the sum of cyclic weights to be
$$ \beta_{n}
       =
   \sum_{\pi \in {\mathfrak S}_{n}} \Wt^{c}(\pi)  . $$
We will write $\beta_{n}(S)$
when the weighting function is associated with avoiding
a certain pattern $S$.

\begin{lemma}
The sum of the cyclic weights $\beta_{n}$ is given by
the integral
$$   \beta_{n}
     =
     n!
 \cdot
     \int_{[0,1]^{n}} \chi_{n+m}(x_{1}, x_{2}, \ldots, x_{n},
          x_{1}, \ldots, x_{m}) \: dx_{1} \cdots dx_{n}. $$
\end{lemma}
\begin{proof}
We evaluate the integral by partitioning the $n$-dimensional
cube by the standard triangulation and noting that
the function $\chi_{n+m}$ is constant on each of the $n!$ simplexes,
that is,
\begin{align*}
&
     \int_{[0,1]^{n}}
         \chi_{n+m}(x_{1}, x_{2}, \ldots, x_{n}, x_{1}, \ldots, x_{m})
              \: dx_{1} \cdots dx_{n}  \\
& =
\sum_{\pi \in {\mathfrak S}_{n}}
     \int_{\onethingatopanother{x \in [0,1]^{n}}{\Pi(x) = \pi}}
       \chi_{n+m}(x_{1}, x_{2}, \ldots, x_{n}, x_{1}, \ldots, x_{m})
                    \: dx_{1} \cdots dx_{n} \\
& =
\sum_{\pi \in {\mathfrak S}_{n}}
     \int_{\onethingatopanother{x \in [0,1]^{n}}{\Pi(x) = \pi}}
           \Wt^{c}(\pi)
                    \: dx_{1} \cdots dx_{n} \\
& = 
\frac{1}{n!} \cdot
\sum_{\pi \in {\mathfrak S}_{n}}  \Wt^{c}(\pi) .
\qedhere
\end{align*}
\end{proof}

\begin{theorem}
Let $(\lambda_{k})_{1 \leq k \leq K}$,
where $K \leq \infty$, be a list of the eigenvalues 
of the operator~$T$
(including multiplicities),
where $T$ has the form given in equation~\eqref{equation_T}.
Let $n \geq m+1$.
Assume that the series $\sum_{k} \lambda_{k}^{n}$
converges absolutely.
Then the
sum of the cyclic weights is given by
$$ \beta_{n} 
        =
      n! \cdot \sum_{k=1}^{K} \lambda_{k}^{n}  . $$
\label{theorem_main}
\end{theorem}
\begin{proof}
Begin to observe that $T^{n}$ is an Hilbert-Schmidt operator.
Since the sum $\sum_{k} |\lambda_{k}|^{n}$
converges then the operator $T^{n}$ is a trace class operator;
see
Section~XI.8, Exercise~49 in~\cite{Dunford_Schwartz}.
The trace of the operator is given by
the convergent sum
$$  \trace(T^{n}) = \sum_{k=1}^{K} \lambda_{k}^{n}  .  $$
But the trace of $T^{n}$ is also given by
the following integral;
see part (c) of the above mentioned exercise in~\cite{Dunford_Schwartz}:
\begin{align*}
\trace(T^{n})
  & = 
\int_{[0,1]^{m}} K(x,x) \: dx \\
  & = 
\int_{[0,1]^{n}}
\chi_{n+m}(x_{1},\ldots, x_{m}, t_{1}, \ldots, t_{n-m},x_{1}, \ldots, x_{m})
\: dx_{1} \cdots dx_{m} dt_{1} dt_{n-m} \\
  & = 
\beta_{n}/n!,
\end{align*}
proving the result.
\end{proof}

We give two examples.
\begin{theorem}
The number of cyclically consecutive $123$-avoiding permutations
in the symmetric group ${\mathfrak S}_{n}$ for $n \geq 2$ is given by
$$   \beta_{n}(123)
     =
       n! \cdot
  \sum_{k=-\infty}^{\infty} \left(\frac{\sqrt{3}}{2\pi (k+1/3)}\right)^{n} . $$
\label{theorem_123}
\end{theorem}
\begin{proof}
Here we let the weight function $\wt$ on ${\mathfrak S}_{3}$
be given by
$\wt(\sigma) = 1 - \delta_{\sigma,123}$,
where $\delta$ denotes the Kronecker delta.
The eigenvalues are given by
$\lambda_{k} = \frac{\sqrt{3}}{2\pi (k+1/3)}$
where $k$ ranges over all the integers;
see~\cite[Proposition~5.1]{Ehrenborg_Kitaev_Perry}.
Note for $n \geq 2$ the series
$\sum_{k=-\infty}^{\infty} |\lambda_{k}|^{n}$ converges.
By Theorem~\ref{theorem_main}
the result follows in the case $n \geq 3$. The case $n=2$
is can be verified by a separate calculation.
\end{proof}

A different approach to prove this theorem is to cyclically shift the 
permutations such that $\pi_{n} = n$.
Now we are looking for $123$-avoiding permutations
$\pi_{1} \cdots \pi_{n-1}$ in ${\mathfrak S}_{n-1}$.
However, we also require that the last two entries yields
a descent, since otherwise the three entries
$\pi_{n-2} \pi_{n-1} \pi_{n}$ would have the $123$ pattern.
The number of such permutations in ${\mathfrak S}_{n-1}$
can be calculated by
$(n-1)! \cdot \pair{T^{n-m}(\kappa)}{\mu}$
where $\kappa$ is the constant function ${\mathbf 1}$
and
$\mu$ is the $0,1$-function $\mu(x,y) = 1$ if $x > y$
and $\mu(x,y) = 0$ if $x < y$.
However, Theorem~2.1 in~\cite{Ehrenborg_Jung}
only gives an asymptotic expansion for the desired
quantity. 
Lastly, one has to do an analytic argument
to show
that the asymptotic expansion actually gives a
convergent series.

As our next example, we consider the weighted example
from~\cite[Section~6]{Ehrenborg_Jung}.
Let $\bb(\pi)$ and $\bbc(\pi)$ denote
the number of double descents of the permutation $\pi$,
respectively,
the number of cyclically double descents of the permutation~$\pi$,
that is,
\begin{align*}
 \bb(\pi)  & = \#\{i \in \{1,2, \ldots, n-2\} \: : \: \pi_{i} < \pi_{i+1} < \pi_{i+2} \} , \\
 \bbc(\pi) & = \#\{i \in \{1,2, \ldots, n\} \: : \: \pi_{i} < \pi_{i+1} < \pi_{i+2} \} , 
 \end{align*}
where the indexes are modulo $n$.

\begin{theorem}
The sum of $2^{\bbc(\pi)}$
over all permutations $\pi$ in ${\mathfrak S}_{n}$
which cyclically do not have a double ascent is given by $n!$,
that is,
$$   \sum_{\onethingatopanother{\pi \in {\mathfrak S}_{n}}
                  {\pi \text{ cyclically $123$-avoiding}}}
         2^{\bbc(\pi)}
      =
           n! . $$
\label{theorem_weight}
\end{theorem}
\begin{proof}
Here the weight function $\wt$ on ${\mathfrak S}_{3}$
be given by
$\wt(132) = \wt(213) = \wt(231) = \wt(312)  = 1$,
$\wt(123) = 0$
but $\wt(321) = 2$.
The associated operator has only one eigenvalue namely $1$;
see~\cite[Theorem~6.1]{Ehrenborg_Jung}.
Hence the result follows directly for $n \geq 3$.
The three cases $n \leq 2$ are direct.
\end{proof}

\begin{corollary}
The sum of $2^{\bb(\pi)}$
over all permutations in ${\mathfrak S}_{n}$
such that $\pi_{1} = n$,
there is no double ascent 
and which end with a descent,
that is, $\pi_{n-1} > \pi_{n}$,
is given by $(n-1)!$.
\end{corollary}
\begin{proof}
The weight of a permutation in the previous theorem
does not change when it is cyclically shifted.
Hence shift the permutation
such that $\pi_{1} = n$ and the sum is given by~$n!/n$.
\end{proof}

As an example, there are $9$ permutations in ${\mathfrak S}_{5}$
that are $123$-avoiding, begin with the element~$5$
and end with a descent.
In lexicographic order they are 
$51432$,
$52143$,
$52431$,
$53142$,
$53241$,
$53421$,
$54132$,
$54231$,
and~$54321$.
Note that the $8$ first permutations all have one double descent
and the last permutation has three double descents yielding 
the sum $8 \cdot 2^{1} + 1 \cdot 2^{3} = 4!$.

It is tempting to use~\cite[Proposition~7.2]{Ehrenborg_Kitaev_Perry}
to obtain a result for 
the number of cyclically consecutive
$123,231,312$-avoiding permutations.
However, note that this proposition only determines a part
of the spectrum of the associated operator corresponding
to an invariant subspace. 
Hence there is not enough information to apply
Theorem~\ref{theorem_main}.

\section{Concluding remarks}

Are there other operators of the form~\eqref{equation_T}
where we can determine the spectrum?
This is not a straightforward question since the spectrum
for $213$-avoiding permutations is given by the equation
$$   \erf\left(\frac{1}{\sqrt{2} \cdot \lambda}\right) = \sqrt{\frac{2}{\pi}}  ,  $$
where $\erf$ denotes the error function.
Is is always true that the sum
$\sum_{k} |\lambda_{k}|^{m}$ is absolutely convergent?
In other words, is the operator $T^{m}$ always trace class?
Does
the generating function
$\sum_{n \geq 0} \beta_{n} \cdot z^{n}/n!$
have a nicer form than
the generating function
$\sum_{n \geq 0} \alpha_{n} \cdot z^{n}/n!$;
see~\cite{Ehrenborg,Elizalde_Noy}.

It is interesting to note that Theorem~\ref{theorem_123}
gives a exact expression for the number of cyclically
consecutive $123$-avoiding permutations.
However, for
consecutive $123$-avoiding permutations
the corresponding result~\cite[Theorem~5.4]{Ehrenborg_Kitaev_Perry}
only yields an asymptotic expansion.

Since this result
of Theorem~\ref{theorem_weight}
is purely combinatorial, it is natural to ask
for a bijective proof.

\section*{Acknowledgments}

The author is grateful to Margaret Readdy 
for her comments on an earlier version of this paper.
The author was partially supported by 
National Security Agency grant~H98230-13-1-0280.

\newcommand{\journal}[6]{{\sc #1,} #2, {\it #3} {\bf #4} (#5), #6.}
\newcommand{\book}[4]{{\sc #1,} ``#2,'' #3, #4.}
\newcommand{\thesis}[4]{{\sc #1,} ``#2,'' Doctoral dissertation, #3, #4.}
\newcommand{\springer}[4]{{\sc #1,} ``#2,'' Lecture Notes in Math.,
                     Vol.\ #3, Springer-Verlag, Berlin, #4.}
\newcommand{\preprint}[3]{{\sc #1,} #2, preprint #3.}
\newcommand{\preparation}[2]{{\sc #1,} #2, in preparation.}
\newcommand{\appear}[3]{{\sc #1,} #2, to appear in {\it #3}}
\newcommand{\submitted}[4]{{\sc #1,} #2, submitted to {\it #3}, #4.}
\newcommand{\JCTA}{J.\ Combin.\ Theory Ser.\ A}
\newcommand{\AdvancesinMathematics}{Adv.\ Math.}
\newcommand{\AdvancesinAppliedMathematics}{Adv.\ in Appl.\ Math.}
\newcommand{\JournalofAlgebraicCombinatorics}{J.\ Algebraic Combin.}

\bigskip

{\em R.\ Ehrenborg,
Department of Mathematics,
University of Kentucky,
Lexington, KY 40506-0027,}
{\tt jrge@ms.uky.edu}

\end{document}